\documentclass[11pt,reqno]{article}

\usepackage[T1]{fontenc}
\usepackage[utf8]{inputenc}
\usepackage[margin=1in]{geometry}
\usepackage{amsmath,amssymb,amsthm,mathtools}
\usepackage{newtxtext,newtxmath}
\usepackage{booktabs}
\usepackage{enumitem}
\usepackage{microtype}
\usepackage{authblk}
\usepackage{xcolor}
\usepackage[hidelinks]{hyperref}

\definecolor{linkblue}{RGB}{22,76,130}
\hypersetup{
  colorlinks=true,
  linkcolor=linkblue,
  citecolor=linkblue,
  urlcolor=linkblue,
  pdftitle={A sharp product bound for disjoint cross-intersecting 3-graphs with covering number three},
  pdfauthor={Arthur F. Ramos, David B. Hulak, Ruy J. G. B. de Queiroz}
}

\newtheorem{theorem}{Theorem}[section]
\newtheorem{lemma}[theorem]{Lemma}
\newtheorem{proposition}[theorem]{Proposition}

\theoremstyle{definition}
\newtheorem{definition}[theorem]{Definition}

\newcommand{\cF}{\mathcal F}
\newcommand{\cG}{\mathcal G}
\newcommand{\cH}{\mathcal H}
\newcommand{\cK}{\mathcal K}
\newcommand{\cB}{\mathcal B}

\newcommand{\V}{V}
\newcommand{\Bext}{\mathsf B}
\newcommand{\Tthree}{\mathsf T_3}

\title{A Sharp Product Bound for Disjoint Cross-Intersecting 3-Graphs\\
with Covering Number Three}
\author[1]{Arthur F. Ramos\thanks{Corresponding author:
\href{mailto:arfreita@microsoft.com}{arfreita@microsoft.com}.}}
\author[2]{David B. Hulak}
\author[3]{Ruy J. G. B. de Queiroz}
\affil[1]{Microsoft}
\affil[2]{Independent Researcher}
\affil[3]{Centro de Inform\'atica, Universidade Federal de Pernambuco}
\date{}

\begin{document}
\maketitle
\vspace{-1.4em}

\begin{abstract}
Lin, Frankl, and Wu proved that the product of the sizes of two
cross-intersecting 3-uniform hypergraphs with covering number three is at most
$121$.  They conjectured that requiring the families to be disjoint lowers
the sharp bound to $100$.  We prove this conjecture.  The smaller family has
at most eleven edges; after fixing it, the other family may be enlarged to its
external family of 3-transversals.  Splitting by matching number leaves an elementary
intersecting case and two finite kernels.  We prove the correctness of both
encodings and every pruning rule, and exhaustive C++20 computations give the
external-blocker maxima $21,19,16,14,12,11,10,9$ for family sizes $4$ through
$11$.  Concrete witnesses are checked independently in Python.  A pair of
ten-edge families on six vertices attains product $100$.
\end{abstract}

\noindent\textbf{Keywords.}
Cross-intersecting hypergraphs; covering number; transversal; blocker;
computer-assisted proof.

\noindent\textbf{2020 Mathematics Subject Classification.}
Primary 05C65; Secondary 05D05, 68R10.

\section{Introduction}

A \emph{$k$-graph} is a finite family of $k$-element sets.  Its covering
number $\tau(\cF)$ is the minimum cardinality of a set meeting every member of
$\cF$.  Two families $\cF$ and $\cG$ are \emph{cross-intersecting} if
$F\cap G\ne\varnothing$ for every $F\in\cF$ and $G\in\cG$.  They are
\emph{disjoint} if $\cF\cap\cG=\varnothing$; their vertex supports may overlap.

For positive integers $k$ and $\ell$, Lin, Frankl, and Wu~\cite{LFW} defined
$m(k,\ell)$ to be the maximum of $|\cF||\cG|$ over cross-intersecting pairs in
which $\cF$ is a $k$-graph with $\tau(\cF)=\ell$ and $\cG$ is an
$\ell$-graph with $\tau(\cG)=k$.  This is a product version of the classical
covering-number problem for intersecting hypergraphs initiated by Erd\H{o}s
and Lov\'asz~\cite{EL}.  Among their exact small-parameter results, Lin,
Frankl, and Wu proved
\[
  m(3,3)=121.
\]
They also exhibited disjoint cross-intersecting 3-graphs of size ten and posed
Conjecture~7.2 of~\cite{LFW}.  Our main result resolves it.

\begin{theorem}\label{thm:main}
Let $\cF$ and $\cG$ be disjoint, cross-intersecting 3-graphs satisfying
$\tau(\cF)=\tau(\cG)=3$.  Then
\[
  |\cF||\cG|\le 100.
\]
The bound is best possible.
\end{theorem}

Our proof is computer-assisted, but the computational part is confined to two
small finite kernels.  The first is a subset dynamic program on 27-bit masks.
The second enumerates 39 core types and partitions a bounded set of incidence
slots.  Neither search assumes an a priori bound on the ambient vertex set.
We give a proof of the encodings, the exact covering-number tests, and every
pruning principle used by the programs.  The complete C++20 source,
certificates, and concrete witnesses are preserved in a versioned
reproducibility archive~\cite{artifact} that does not contain this manuscript.

The proof uses $m(3,3)=121$ only to show that the smaller of $\cF$ and $\cG$
has at most eleven edges.  In particular, the case of eleven edges is checked
directly; we do not use the equality classification from~\cite{LFW}.

\section{External blockers and the basic reduction}

All hypergraphs are regarded as living on their support
\[
  \V(\cH)=\bigcup_{H\in\cH}H.
\]
This convention ensures that the transversal families below are finite.

\begin{definition}
If $\cH$ is a 3-graph with $\tau(\cH)=3$, define
\[
  \Tthree(\cH)
    =\{T\subseteq \V(\cH): |T|=3\text{ and }T\cap H\ne\varnothing
      \text{ for all }H\in\cH\}
\]
and define its \emph{external 3-blocker} by
\[
  \Bext(\cH)=\Tthree(\cH)\setminus\cH.
\]
\end{definition}

\begin{lemma}[External-blocker reduction]\label{lem:reduction}
Let $\cF,\cG$ satisfy the hypotheses of Theorem~\ref{thm:main}, and suppose
$|\cG|\le |\cF|$.  Put $b=|\cG|$ and $\cB=\Bext(\cG)$.  Then
\[
  b\le 11,\qquad \cF\subseteq\cB,
  \qquad \tau(\cG)=\tau(\cB)=3,
\]
and consequently $|\cF||\cG|\le b|\cB|$.
\end{lemma}

\begin{proof}
The theorem $m(3,3)=121$ gives
\[
  b^2\le |\cF||\cG|\le 121,
\]
so $b\le 11$.  Let $F\in\cF$.  Cross-intersection says that $F$ meets every
edge of $\cG$.  Moreover, $F\subseteq\V(\cG)$: if
$x\in F\setminus\V(\cG)$, then $x$ belongs to no edge of $\cG$, so the
2-set $F\setminus\{x\}$ covers $\cG$, contradicting $\tau(\cG)=3$.
Consequently $F\in\Tthree(\cG)$, while disjointness says
$F\notin\cG$.  Thus $\cF\subseteq\cB$.

Any cover of $\cB$ covers its subfamily $\cF$, and therefore
$\tau(\cB)\ge\tau(\cF)=3$.  Conversely, every edge of $\cG$ meets every
member of $\cB$ by the definition of a transversal.  Hence each edge of
$\cG$ is a 3-cover of $\cB$, giving $\tau(\cB)\le3$.
\end{proof}

It remains to prove
\begin{equation}\label{eq:blocker-target}
  b|\Bext(\cG)|\le 100
\end{equation}
for $3\le b\le11$.  The lower endpoint is automatic: a family with fewer
than three edges has a cover of size at most two.  If $b=3$, then the three
edges are pairwise disjoint, since a common point of two edges together with
one point from the third would give a 2-cover.  A 3-transversal chooses one
point from each of three 3-sets, so $|\Bext(\cG)|\le27$ and
$b|\Bext(\cG)|\le81$.

Write $\nu(\cG)$ for the matching number.  Since every cover meets each edge
of a matching,
\[
  1\le \nu(\cG)\le\tau(\cG)=3.
\]
We treat the three possible matching numbers separately.

\section{The intersecting branch}

The case $\nu(\cG)=1$ follows from a direct transversal count.

\begin{lemma}\label{lem:intersecting}
If $\cH$ is an intersecting 3-graph with $\tau(\cH)=3$, then
\[
  |\Tthree(\cH)|\le19.
\]
\end{lemma}

\begin{proof}
Fix $A=\{a_1,a_2,a_3\}\in\cH$.  We first count transversals containing at
least two points of $A$.  For each pair $P\subset A$, choose an edge avoiding
$P$, which is possible because $P$ is not a cover.  Since $\cH$ is
intersecting, this edge contains the third point of $A$ and at most two points
outside $A$.  A transversal extending $P$ must choose its third point from
that edge.  Apart from $A$ itself, this gives at most two extensions for each
of the three pairs, hence at most seven transversals.

Now fix $a_i$ and consider transversals meeting $A$ exactly in $a_i$.  Choose
an edge $D$ avoiding $a_i$.  It meets $A\setminus\{a_i\}$ and therefore has
at most two points outside $A$.  The transversal must contain one of these
outside points, say $x$.  The pair $\{a_i,x\}$ is not a cover, so there is an
edge avoiding it.  That edge again meets $A\setminus\{a_i\}$ and offers at
most two outside choices for the final point.  Thus there are at most four
transversals for each $a_i$, or twelve in total.  Adding the two classes gives
$7+12=19$.
\end{proof}

Every edge of an intersecting 3-graph is itself a transversal.  Hence, for
$4\le b\le11$,
\[
  |\Bext(\cG)|\le19-b,
  \qquad b|\Bext(\cG)|\le b(19-b)\le90.
\]

\section{The common-union lemma}

The two nonintersecting branches use the following observation from
Lin--Frankl--Wu~\cite{LFW}, for which we include the short proof.

\begin{lemma}[Common union]\label{lem:union}
If 3-graphs $\cH$ and $\cK$ are cross-intersecting and satisfy
$\tau(\cH)=\tau(\cK)=3$, then $\V(\cH)=\V(\cK)$.
\end{lemma}

\begin{proof}
Suppose that $x\in H\in\cH$ but $x\notin\V(\cK)$.  Since $H$ meets every
edge of $\cK$, the two-set $H\setminus\{x\}$ also meets every edge of $\cK$,
contrary to $\tau(\cK)=3$.  The reverse inclusion is symmetric.
\end{proof}

We apply the lemma to $\cG$ and $\Bext(\cG)$, which are cross-intersecting by
definition and have covering number three by Lemma~\ref{lem:reduction}.

\section{The branch of matching number three}

Assume $\nu(\cG)=3$, and fix three pairwise disjoint edges $A,B,C$ of
$\cG$.  Any 3-transversal meets each of these edges exactly once, so it is a
point of the grid
\[
  \mathcal Q=A\times B\times C,
  \qquad |\mathcal Q|=27.
\]
In particular, $\Bext(\cG)\subseteq\mathcal Q$.  The common-union lemma gives
\[
  \V(\cG)=\V(\Bext(\cG))\subseteq A\cup B\cup C,
\]
so the entire family lies on nine vertices.

Besides $A,B,C$, there are
\[
  \binom{9}{3}-3=81
\]
possible edges.  For each such edge $E$, define its \emph{kill mask}
$K(E)\subseteq\mathcal Q$ by
\[
  K(E)=\{Q\in\mathcal Q: Q\cap E=\varnothing\}
       \cup \bigl(\{E\}\cap\mathcal Q\bigr).
\]
The first term removes grid points that fail to meet $E$; the second removes
$E$ itself because the blocker is external.  If $E_1,\dots,E_{b-3}$ are the
additional edges of $\cG$, then
\begin{equation}\label{eq:grid-blocker}
  \Bext(\cG)=\mathcal Q\setminus
  \bigl(K(E_1)\cup\cdots\cup K(E_{b-3})\bigr).
\end{equation}

Equation~\eqref{eq:grid-blocker} gives a compact exhaustive search.  After
scanning some of the 81 candidate edges, the dynamic-programming state
$(r,M)$ records that the union mask $M$ is obtainable by choosing exactly $r$
distinct scanned edges.  Updates run in descending $r$, preventing reuse of
an edge.  Equal masks are merged because the blocker, its cardinality, and its
covering number depend only on $M$.

For a target $t$, a state is retained only if
\[
  |M|\le26-t.
\]
Indeed, a counterexample to $|\Bext(\cG)|\le t$ would have at most $26-t$
killed grid points.  Since kill masks only grow under union, a discarded state
cannot later become a counterexample.  At depth $b-3$, the program forms the
surviving grid family and tests all pairs of the nine vertices.  The surviving
family has covering number three exactly when none of these pairs covers it;
the same test also excludes a cover of size zero or one.

This proves the following certificate statement.

\begin{proposition}\label{prop:matching-three-certificate}
For fixed $b\in\{4,\dots,11\}$, the program
\texttt{check\_matching\_three} with target $t$ returns success only after
representing every external-blocker mask arising from a $b$-edge family in the
matching-number-three branch that has covering number three and could have size
greater than $t$.  Consequently a successful run proves
$|\Bext(\cG)|\le t$ in this branch.
\end{proposition}

\begin{proof}
The nine-vertex reduction is exhaustive by Lemma~\ref{lem:union}.  The list of
81 candidates contains every possible additional edge exactly once, and the
descending subset dynamic program represents every choice of $b-3$ distinct
candidates.  Equation~\eqref{eq:grid-blocker} proves the mask semantics.  The
only early deletion is the monotone cardinality test above, and the final
pair test is equivalent to covering number three.  Thus no family satisfying
the hypotheses and violating the target is omitted.
\end{proof}

\section{The branch of matching number two}

Assume $\nu(\cG)=2$.  Fix disjoint edges
\[
  A=\{0,1,2\},\qquad B=\{3,4,5\},\qquad U=A\cup B.
\]
Write $R=b-2$, so $2\le R\le9$.  Every other edge meets $U$, since an edge
disjoint from $U$ would extend $A,B$ to a 3-matching.

\subsection{Core types and incidence rows}

For an additional edge $E$, its \emph{core} is $E\cap U$.  The core is
nonempty, has size at most three, and is neither $A$ nor $B$.  There are
\[
  \binom61+\binom62+\binom63-2=39
\]
possible core types.  A core of size $s$ has $3-s$ outside-incidence slots.

Fix a multiset of $R$ core types.  Every vertex outside $U$ is encoded by a
nonempty row $S\subseteq[R]$, consisting of the additional edges that contain
that vertex.  The rows partition the outside-incidence slots: column $i$
occurs in exactly $3-|E_i\cap U|$ rows.  Conversely, the core multiset and any
such row partition reconstruct a 3-graph containing the fixed 2-matching.
Duplicate edges are rejected, and at a completed leaf the program explicitly
checks that no three edges are pairwise disjoint.

This representation removes any dependence on the size of an ambient ground
set.  There are at most $2R\le18$ outside slots, hence at most 18 relevant
outside vertices.

\subsection{An exact formula for the external blocker}

Let
\[
  \mathcal P=\{\{a,b\}:a\in A,\ b\in B\}
\]
be the nine cross-pairs.  For $p\in\mathcal P$, let $D_p\subseteq[R]$ be the
additional-edge columns whose cores avoid $p$.  If an outside vertex $v$ has
row $S$, then
\begin{equation}\label{eq:row-condition}
  p\cup\{v\}\text{ meets every edge of }\cG
  \quad\Longleftrightarrow\quad D_p\subseteq S.
\end{equation}
The triple is placed in the external blocker unless it is itself an edge of
$\cG$.  Notice that every blocker member not contained in $U$ must have
exactly one point in $A$, one in $B$, and one outside $U$.

The remaining blocker members lie entirely in $U$.  Let $\mathcal I$ be the
set of 3-subsets of $U$ that meet $A$, $B$, and every additional core, with
the internal edges of $\cG$ removed.  Then $\mathcal I$ is exactly the
internal part of $\Bext(\cG)$, and~\eqref{eq:row-condition} describes the
external part.  These two parts are disjoint, so the program's score
\begin{equation}\label{eq:blocker-score}
  |\mathcal I|+
  \sum_{\text{outside rows }S}
  \bigl|\{p\in\mathcal P:D_p\subseteq S,
            \ p\cup\{v_S\}\notin\cG\}\bigr|
\end{equation}
equals $|\Bext(\cG)|$ exactly.

\subsection{The exact covering-number test}

Every 2-set has one of the following forms:
\begin{enumerate}[label=(\roman*),leftmargin=2.2em]
  \item two vertices of $U$;
  \item one vertex of $U$ and one outside vertex;
  \item two outside vertices.
\end{enumerate}
For a core pair $q\subset U$, either an internal blocker member avoids $q$, or
an external blocker member $p\cup\{v\}$ avoids it, which occurs precisely when
$p\cap q=\varnothing$.  This gives the exact test for (i).

Fix $u\in U$.  If $u$ meets every member of $\mathcal I$, then for every
outside row $S$ the pair $\{u,v_S\}$ fails to cover the blocker exactly when
some \emph{other} row contributes an external blocker member whose cross-pair
avoids $u$.  Thus at least two distinct rows must contribute such members.
This gives the exact test for (ii).  If $\mathcal I$ is nonempty, every pair
of outside vertices already misses an internal blocker member.  If
$\mathcal I$ is empty, a pair of outside vertices fails to cover exactly when
some third outside row contributes a blocker member.  Hence at least three
distinct rows must contribute.  This gives the exact test for (iii).

The three conditions are implemented by the leaf predicate
\texttt{tauExact}.  They test every possible 2-cover and therefore are
equivalent to $\tau(\Bext(\cG))=3$.

\subsection{Exhaustive generation and safe pruning}

Core multisets are generated in sorted order.  The program keeps only the
lexicographically least representative under the 72 automorphisms obtained by
permuting $A$, permuting $B$, and optionally swapping the two fixed edges.
If a sorted prefix is not least in its orbit, adding further elements cannot
make the complete sorted multiset least: after applying the improving
automorphism, the first entries of the completed sorted image are no larger
than the improved prefix.  Prefix canonicalization is therefore safe.

For a fixed core multiset, row partitions are generated canonically by taking
the least column with unused capacity, forming every possible row containing
it, and ordering repeated rows with the same least column.  This is the
standard restricted-growth recursion for a multiset of incidence slots and
produces every row partition.

The main numerical upper bound is
\begin{equation}\label{eq:cheap-bound}
  |\mathcal I|+
  \sum_{p\in\mathcal P}
  \min_{i\in D_p}(3-|E_i\cap U|).
\end{equation}
For fixed $p$, every row contributing the blocker triple $p\cup\{v\}$ must
contain every column in $D_p$.  The number of such distinct rows is therefore
at most the smallest capacity of one of those columns.  This proves
\eqref{eq:cheap-bound}.  If some $D_p$ is empty, then $p$ is a 2-cover of
$\cG$, so the core multiset may be rejected.  During row generation, the same
bound is recomputed using remaining capacities.

The remaining feasibility pruning mirrors the three pair tests above.  For
every blocker witness still required to defeat a possible 2-cover, the
program checks whether the remaining column capacities could supply one, two,
or three future rows.  It ignores some compatibility requirements, so it is
optimistic: a positive answer retains the node, while a negative answer proves
that no completion can satisfy the leaf test.  Thus this pruning cannot remove
a counterexample.

We obtain the matching-number-two analogue of
Proposition~\ref{prop:matching-three-certificate}.

\begin{proposition}\label{prop:matching-two-certificate}
For fixed $b\in\{4,\dots,11\}$, the program
\texttt{check\_matching\_two} with target $t$ returns success only after
representing every $b$-edge family in the matching-number-two branch whose
external blocker has covering number three and could have size greater than
$t$.  Consequently a successful run proves $|\Bext(\cG)|\le t$ in this
branch.
\end{proposition}

\begin{proof}
Every additional edge has one of the 39 core types, and every outside vertex
defines an incidence row, so the encoding is exhaustive.  The symmetry test
retains an orbit representative, and the row recursion generates every
partition of the incidence slots.  Duplicate edges and families with a
3-matching are tested explicitly at complete leaves.  Formula
\eqref{eq:blocker-score} gives the exact blocker size, and the three-form
analysis gives the exact covering-number predicate.  Bound
\eqref{eq:cheap-bound} and its residual version cannot underestimate a
completion.  The feasibility test removes a node only when the remaining
capacities cannot furnish a witness required by the exact leaf predicate.
Therefore no family satisfying the branch hypotheses and violating the target
is omitted.
\end{proof}

\section{The finite certificate}

Both kernels were compiled as C++20 programs using only the standard library.
The theorem run checks each value $b=4,\dots,11$ with target
$\lfloor100/b\rfloor$.  The exactness run repeats each case at the bound shown
below and then lowers the target by one; exit status 2 records an attained
state or family at the displayed value.

\begin{table}[ht]
  \centering
  \caption{Exact external-blocker maxima, attained separately in each
  nonintersecting branch.}
  \label{tab:bounds}
  \setlength{\tabcolsep}{7pt}
  \begin{tabular}{c@{\qquad}rrrrrrrr}
    \toprule
    $b=|\cG|$ & 4 & 5 & 6 & 7 & 8 & 9 & 10 & 11 \\
    \midrule
    $\max|\Bext(\cG)|$ & 21 & 19 & 16 & 14 & 12 & 11 & 10 & 9 \\
    $b\,|\Bext(\cG)|$ & 84 & 95 & 96 & 98 & 96 & 99 & 100 & 99 \\
    \bottomrule
  \end{tabular}
\end{table}

For the largest matching-number-two upper-bound case, the run processed
$9{,}813{,}493$ canonical core-sequence leaves and $55{,}133{,}110$ row-search
nodes.  For the largest matching-number-three upper-bound case, the final
dynamic-programming layer contained $31{,}023$ masks after threshold pruning.
The committed transcripts record all cases, and a SHA-256 manifest binds the
programs, drivers, documentation, and tests.  Whenever the lowered-threshold
run finds an attained value, the kernel emits a concrete family.  A separate
Python program recomputes the matching number, the covering numbers of both
the family and its external blocker, and the blocker cardinality for all
sixteen witnesses without sharing code with either exhaustive kernel.

Table~\ref{tab:bounds}, together with the intersecting estimate, proves
\eqref{eq:blocker-target} for every $3\le b\le11$.  Lemma~\ref{lem:reduction}
therefore proves the upper bound in Theorem~\ref{thm:main}.

\section{Sharpness}

Let the vertex set be $[6]$ and define
\begin{align*}
  \cF={}&\{123,456,124,356,125,346,134,256,136,245\},\\
  \cG={}&\{126,345,135,246,145,236,146,235,156,234\}.
\end{align*}
Here, for instance, $123$ abbreviates $\{1,2,3\}$.  The twenty triples of
$[6]$ split into ten complementary pairs.  The family $\cF$ is the union of
five complementary pairs and $\cG$ is the union of the other five.  Thus they
are disjoint.  Two triples of $[6]$ are disjoint exactly when they are
complements, and complements were placed in the same family, so the two
families are cross-intersecting.

The displayed lists also show that every 2-subset of $[6]$ is contained in a
member of each family.  Its complementary edge belongs to the same family and
is disjoint from that 2-set.  Hence no 2-set covers either family.  On the
other hand, every edge of one family is a 3-cover of the other by
cross-intersection.  It follows that $\tau(\cF)=\tau(\cG)=3$.  Therefore
\[
  |\cF||\cG|=10\cdot10=100,
\]
which proves sharpness.  The artifact includes a separate Python program that
verifies all of these assertions without sharing code with either exhaustive
kernel.

\section{Reproducibility and scope of the computation}

The complete certificate is reproduced by
\begin{center}
  \texttt{make check}\qquad\texttt{make test}\qquad\texttt{make exact}.
\end{center}
The first command compiles both kernels and proves all upper-bound cases.  The
second independently verifies the sharp construction and all sixteen concrete
exact-bound witnesses.  The third verifies the exact values in
Table~\ref{tab:bounds}.  The programs use fixed integer and bit-mask arithmetic;
there is no randomness, floating-point arithmetic, external solver, or
unchecked input data.  The largest case was also run under AddressSanitizer and
UndefinedBehaviorSanitizer.

The computer-assisted claim is deliberately narrow.  The human argument
reduces the theorem to a finite statement and proves that each program
enumerates that statement without unsafe pruning.  The machine then supplies
the finite exhaustion.  A proof-assistant formalization or an independent
implementation would provide a stronger assurance level, but neither is an
additional mathematical hypothesis of the certificate presented here.

\section*{Data and code availability}

Version 1.0.0 of the complete software certificate, including source code,
drivers, run transcripts, concrete witnesses, independent witness tests, and a
SHA-256 manifest, is archived on Zenodo~\cite{artifact}.  The archive contains
no manuscript source or PDF.  It is released under the MIT License.
Its version-specific persistent identifier is
\href{https://doi.org/10.5281/zenodo.21881247}
{doi:10.5281/zenodo.21881247}.

\section*{Acknowledgements and declarations}

\noindent\textbf{Acknowledgements.}
The authors thank Long Lin, Peter Frankl, and Hehui Wu for formulating the
problem and making their preprint available.

\medskip
\noindent\textbf{Author contributions.}
Arthur F. Ramos: proof development, software, validation, and writing---original
draft. David B. Hulak and Ruy J. G. B. de Queiroz: conceptualization, formal
analysis, and writing---review and editing.

\medskip
\noindent\textbf{Funding.}
This research received no specific grant from funding agencies in the public,
commercial, or not-for-profit sectors.

\medskip
\noindent\textbf{Competing interests.}
The authors declare no competing interests.

\section*{Declaration of generative AI and AI-assisted technologies}

During the preparation of this work, the authors used OpenAI Codex to assist
with proof exploration, implementation, test design, and manuscript editing.
The authors reviewed and edited the resulting material. Arthur F. Ramos reran
the deterministic certificate. The authors take full responsibility for the
content of the article.

\end{document}